\documentclass[a4paper,12pt,reqno]{amsart}

\usepackage{amsmath}
\usepackage{amssymb}
\usepackage{amsfonts}
\usepackage{graphicx}
\usepackage[colorlinks]{hyperref}
\renewcommand\eqref[1]{(\ref{#1})} 

\graphicspath{ {images/} }
\title[A survey of Hardy type inequalities on homogeneous groups]{A survey of Hardy type inequalities on homogeneous groups}

\author[D. Suragan]{Durvudkhan Suragan}
\address{
	Durvudkhan Suragan:
	\endgraf
	Department of Mathematics
	\endgraf
 Nazarbayev University
	\endgraf
	53 Kabanbay Batyr Ave, Astana 010000
	\endgraf
	Kazakhstan
	\endgraf
	{\it E-mail address} {\rm durvudkhan.suragan@nu.edu.kz}}


\thanks{The author was supported in parts by the Nazarbayev University program 091019CRP2120 and the Nazarbayev University grant 240919FD3901. No new data was collected or generated during the course of this research.}

\newtheoremstyle{theorem}
{10pt}          
{10pt}  
{\sl}  
{\parindent}     
{\bf}  
{. }    
{ }    
{}     
\theoremstyle{theorem}

\numberwithin{equation}{section}
\theoremstyle{plain}
\newtheorem{thm}{Theorem}[section]

\newtheorem{lem}[thm]{Lemma}

\theoremstyle{definition}
\newtheorem{defn}[thm]{Definition}

\newtheoremstyle{defi}
{10pt}          
{10pt}  
{\rm}  
{\parindent}     
{\bf}  
{. }    
{ }    
{}     
\theoremstyle{defi}

\begin{document}
		\begin{abstract}
	In this review paper, we survey Hardy type inequalities from the point of view of Folland and Stein's homogeneous groups. Particular attention is paid to Hardy type inequalities on  stratified groups which give a special class of homogeneous groups. In this environment, the theory of Hardy type inequalities becomes intricately intertwined with the properties of sub-Laplacians and more general subelliptic partial differential equations. Particularly, we discuss the Badiale-Tarantello conjecture and a conjecture on the geometric Hardy inequality in a half-space of the Heisenberg group with a sharp constant.  
	\end{abstract}
	\maketitle
	
\section{Introduction}

	In 1918, G. H. Hardy proved an inequality (discrete and in one variable) \cite{Hardy} now bearing his
	name, which in $\mathbb{R}^{n}$ can be formulated as
	$$
	\left\|\frac{f}{{ \|x\|}}\right\|_{L^{2}(\mathbb{R}^{n})}\leq { \frac{2}{n-2}}\|\nabla f\|_{L^{2}(\mathbb{R}^{n})}, \;\;n\geq 3,
	$$
	where $\nabla$ is the standard gradient in $\mathbb{R}^{n}$, $f\in C_{0}^{\infty}(\mathbb{R}^{n})$, $\|x\|$ is the Euclidean norm, and the constant $\frac{2}{n-2}$ is known to be sharp.
	Note that the multidimensional version of the Hardy inequality was proved by J. Leray \cite{Leray}. 
	
	The Hardy inequality has many applications in the analysis of linear and nonlinear PDEs, for example, in existence and nonexistence results for second order partial differential equations of the form:
	
	$$
	\begin{array}{c}{0} \\ {u_{t}} \\ {u_{t t}}\end{array} \Bigg\}-\Delta u={ \lambda} \frac{|u|^{s}}{{\|x\|^{2}}}.
	$$
	The criteria on existence (or nonexistence) of a solution depends on a relation between the constants $\lambda$ and $ \frac{2}{n-2}$.
	
	In the equation, instead of $ \frac{1}{\|x\|^{2}}$ (Hardy potential) one may have more general function which motivates to study   weighted versions of the Hardy inequality. On the other hand, one may consider some different operators instead of the Laplacian, say, the $p$-Laplacian.
	
	The $L^p$-version of the Hardy inequality (used e.g. for $p$-Laplacian) takes the form
	\begin{equation}\nonumber
	\left\|\frac{f}{\|x\|}\right\|_{L^{p}(\mathbb{R}^{n})}\leq \frac{p}{n-p}\|\nabla f\|_{L^{p}(\mathbb{R}^{n})}, \;\;1< p<n,
	\end{equation}
	where the constant $\frac{p}{n-p}$ is sharp.
	There are already many excellent presentations on the (classical) Hardy inequalities and their extensions, see e.g. \cite{BEL15}, \cite{Davies},  \cite{Mazya}, and \cite{Opic} as well as references therein. 
	
	The purpose of the present review paper is to offer a brief survey of Hardy type inequalities on homogeneous groups. These groups give one of most general classes of  noncommutative nilpotent Lie groups. The intersection of analysis of Hardy type inequalities and theory of homogeneous groups is a beautiful area of mathematics with links to many other subjects.

	Since L. H\"ormander's fundamental work \cite{Hor} the operators of the type sum of squares of vector fields have been studied intensively, and today's literature on the subject is quite large. Much of the development in the field has been connected to the development of analysis on the homogeneous groups, following the ideas of E. Stein's talk at ICM 1970 \cite{Stein}. Since then continuing through the rest of his life, a substantial part of E. Stein's research is related to  analysis on homogeneous groups  (see \cite{Folland19}). 
		
	Among many, one of the motivations behind doing analysis on the homogeneous groups is the {``distillation of ideas and results of harmonic analysis depending only on the group and dilation structures"} \cite{FS1}.

		In the 1990s, a lot of work concerning Hardy inequalities was already developed in the context of the elliptic operators, but not very much had been done in the framework of (nonelliptic) subellipticity, in particular, for the Heisenberg sub-Laplacians.

	Note that the sub-Laplacian on the nilpotent Lie groups are (left invariant homogeneous) subelliptic differential operators and it is known that it is elliptic if and only if the Lie group is Abelian (Euclidean).
	
	 N. Garofalo and E. Lanconelli \cite{GL} have an important contribution to the development of Hardy inequalities on the Heisenberg group with their original approach, which is based on a far-reaching use of the fundamental solutions. Later this idea was extended to general stratified groups by several authors. There is also another approach in the theory of Hardy inequalities on stratified groups, the so-called horizontal estimates, which was suggested by L. D'Ambrosio on the Heisenberg group \cite{Ambrosio04} and on stratified groups \cite{Ambrosio05}. We give further discussions in this direction in Section \ref{Sec2}. 
	 The general theory of Hardy type inequalities in the setting of homogeneous groups is reviewed in Section \ref{Sec3}.
	 
	  These are notes mainly from my lecture at the 12th ISAAC Congress in Aveiro (2019). The lecture was partially based on our recent open access book with the title ``Hardy inequalities on homogeneous groups" with Michael Ruzhansky \cite{RS_book}. 
	 
	 \section{Hardy type inequalities on stratified groups}\label{Sec2}
	 A (connected and simply connected) Lie group $\mathbb{G}$ is { graded} if its Lie algebra: 
	 $$\mathfrak{g}=\bigoplus_{i=1}^{\infty}\mathfrak{g}_{i},$$
	 where $\mathfrak{g}_{1}, \mathfrak{g}_{2},...,$ are vector subspaces of $\mathfrak{g}$, only finitely many  not  $\{0\}$, and
	 $$[\mathfrak{g}_{i},\mathfrak{g}_{j}]\subset \mathfrak{g}_{i+j} \;\;\forall i, j\in \mathbb{N}.$$
	 
	 If $\mathfrak{g}_{1}$ generates the Lie algebra $\mathfrak{g}$ through commutators, the group is said to be {stratified} (see, e.g. \cite{FR}).
	 
	 {\bf Example 1 (Abelian case)}. The Euclidean group $(\mathbb{R}^{n},+)$ is graded: its Lie algebra $\mathbb{R}^{n}$ is trivially graded. Obviously, it is also stratified.  
	 
	 {\bf Example 2 (Heisenberg group)}. The Heisenberg group $\mathbb{H}^{n}$ is stratified: its Lie algebra $\mathfrak{h}_{n}$ can be decomposed as $\mathfrak{h}_{n}=\mathfrak{g}_{1}\oplus \mathfrak{g}_{2}$ where $\mathfrak{g}_{1}=\oplus_{j=1}^{n}\mathbb{R}X_{j}\oplus \mathbb{R}Y_{j}$ and $\mathfrak{g}_{2}=\mathbb{R}T$, where
	 \begin{equation}\nonumber
	 X_{j}=\partial_{x_{j}}-\frac{y_{j}}{2}\partial_{t},\;\;Y_{j}=\partial_{y_{j}}+
	 \frac{x_{j}}{2}\partial_{t},\;\;j=1,\ldots,n,\;\;T=\partial_{t}.
	 \end{equation}
	
Note that the concept of stratified groups was introduced for the first time by Gerald Folland in 1975 \cite{Folland75}. However,  in the literature of sub-Riemannian geometry, stratified groups are commonly called (homogeneous) Carnot groups.
	
	Let $\mathbb{G}$ be a stratified group, i.e. there is $\mathfrak{g}_{1}\subset \mathfrak{g}$ (the first stratum), with its basis $X_1,\ldots,X_N$ generating its Lie algebra $\mathfrak{g}$ through their commutators. Then the sub-Laplacian
	$$\mathcal{L}:=X_1^2+\cdots+X_N^2$$
	is subelliptic, and $\nabla_{H}=(X_1,\ldots,X_n)$ is the so-called horizontal gradient.
	
	\medskip
	
	Folland proved that $\mathcal{L}$ has a fundamental solution $=Cd(x)^{2-Q}$ for some homogeneous quasi-norm $d(x)$ called the $\mathcal{L}$-gauge. Here $Q$ is homogeneous dimension of $\mathbb{G}$.
	
	As we mentioned in the introduction, by using the fundamental solution (i.e. the $\mathcal{L}$-gauge) of the sub-Laplacian $\mathcal{L}$,  in 1990 Garofalo and Lanconelli \cite{GL} proved the Hardy inequality on the Heisenberg group \footnote{In the case of the Heisenberg group, $\mathcal{L}$-gauge is called a Kaplan distance.}, then in 2005 L. D'Ambrosio \cite{Ambrosio05} and independently in 2008 Goldstein and Kombe \cite{GK08} established the Hardy inequality on the general stratified groups:
	\begin{equation}\label{fundsol}
	\left\|\frac{f}{d(x)}\right\|_{L^{2}(\mathbb{G})}\leq \frac{2}{Q-2}\|\nabla_{H} f\|_{L^{2}(\mathbb{G})}, \;\;Q\geq 3.
	\end{equation}
In \cite{RS17Kac}, we obtained the Hardy inequality on the general stratified groups with boundary terms.  In general, this approach can be called a \emph{fundamental solution approach}.
On the other hand, there is another approach to obtain the Hardy inequality on the stratified groups which is a \emph{horizontal estimate approach}.  	
In 2004, L. D'Ambrosio proved a ``horizontal" version of the Hardy inequality on the Heisenberg  group \cite{Ambrosio04} (see also \cite{Ambrosio05}). In 2017, we extended this result for the general stratified groups \cite{Ruzhansky-Suragan:JDE} in the form 
	\begin{equation}\label{horest}
	\left\|\frac{f}{|x'|}\right\|_{L^{2}(\mathbb{G})}\leq \frac{2}{N-2}\|\nabla_{H} f\|_{L^{2}(\mathbb{G})}, \;\;N\geq 3,
	\end{equation}
 with the Euclidean norm $|x'|$ on the first stratum and $x=(x',x'')$. 	
 Here $N$ is (topological) dimension of the first stratum of $\mathbb{G}$.
	The $L^p$-version of \eqref{fundsol} on the Heisenberg group and the stratified groups were proved by using  different methods by Niu, Zhang and Wang in 2001 \cite{NZW}, L. D'Ambrosio in 2004 \cite{Ambrosio04} and in 2005 \cite{Ambrosio05}, 
	Adimurthi and Sekar in 2006 \cite{Adimurthi}, Danielli, Garofalo and Phuc in 2011 \cite{DGP11}, as well as Jin and Shen in 2011 \cite{JS11}:
	\begin{equation}\label{fundsolLp}
	\left\|\frac{f}{d(x)}\right\|_{L^{p}(\mathbb{G})}\leq \frac{p}{Q-p}\|\nabla_{H} f\|_{L^{p}(\mathbb{G})}, \;\;Q\geq 3,\;\;1<p<Q.
		\end{equation}
	Moreover, the $L^p$-version of \eqref{horest} was proved by L. D'Ambrosio on the Heisenberg group \cite{Ambrosio04} and then it  extended on the general stratified groups in \cite{Ambrosio05}. We obtained its extention on the general stratified groups \cite{Ruzhansky-Suragan:JDE} by a different method in the form
	\begin{equation}\label{horestLp}
	\left\|\frac{f}{|x'|}\right\|_{L^{p}(\mathbb{G})}\leq \frac{p}{N-p}\|\nabla_{H} f\|_{L^{p}(\mathbb{G})}, \;\;Q\geq 3,\;\;1<p<N.
	\end{equation}
	Both constants in \eqref{fundsolLp} and \eqref{horestLp} are sharp for functions from, say, $C_{0}^{\infty}(\mathbb{G})$. In particular, a special case of the horizontal estimate implies the Badiale-Tarantello conjecture.

{\bf Badiale-Tarantello conjecture}:
Let $x=(x',x'')\in\mathbb{R}^{N}\times\mathbb{R}^{n-N}.$
Badiale and Tarantello \cite{BT02} proved that for  $1< p<N\leq n$ there exists a constant $C_{n,N,p}$ such that
\begin{equation}\label{BT}
\left\|\frac{1}{|x'|}f
\right\|_{L^{p}(\mathbb{R}^{n})}\leq C_{n,N,p}\left\|\nabla f\right\|_{L^{p}(\mathbb{R}^{n})}.
\end{equation}
Clearly, for $N=n$ this gives the classical Hardy inequality with
the best constant
$$C_{n,p}=\frac{p}{n-p}.$$
It was conjectured by Badiale and Tarantello that
the best constant in \eqref{BT} is given by
\begin{equation}
C_{N,p}=\frac{p}{N-p}.
\end{equation}
This conjecture was proved by Secchi, Smets and Willen in 2003 \cite{SSW03}. 

As a consequence of the horizontal estimate techniques, we also gave a new proof of the Badiale-Tarantello conjecture in  \cite{Ruzhansky-Suragan:JDE}.

 {\bf Hardy inequality in half-space}:
		Let us recall the Hardy inequality in a half-space of $\mathbb{R}^n$:
		\begin{equation}\label{HS1}
		\left( \frac{p-1}{p}\right)^p \int_{\mathbb{R}^n_+} \frac{|u|^p}{x_n^p}dx\leq 	\int_{\mathbb{R}^n_+} |\nabla u|^p dx,\quad p>1,
		\end{equation}
		for every function $u \in C_0^{\infty}(\mathbb{R}^n_+)$, where $\nabla$ is the usual Euclidean gradient and 
		$$\mathbb{R}^n_+:= \{(x',x_n)|x':=(x_1,\ldots,x_{n-1})\in\mathbb{R}^{n-1}, x_n>0 \},\, n \in \mathbb{N}.$$

		There is a number of studies related to inequality \eqref{HS1} by Maz'ya, Davies, Opic-Kufner (see, e.g. \cite{Davies},  \cite{Mazya}, and \cite{Opic}) and others.

		Filippas, Maz'ya and Tertikas in 2007 \cite{FMT07} established the Hardy-Sobolev inequality in the following form
		\begin{equation}\label{HS2}
		C \left( \int_{\mathbb{R}^n_+} |u|^{p^*} dx\right)^{ \frac{1}{p^*}} \leq 	\left(\int_{\mathbb{R}^n_+} |\nabla u|^p dx - \left( \frac{p-1}{p}\right)^p \int_{\mathbb{R}^n_+} \frac{|u|^p}{x_n^p}dx\right)^{\frac{1}{p}},
		\end{equation}
		for all function $u \in C_0^{\infty}(\mathbb{R}^n_+)$, where $p^*=\frac{np}{n-p}$ and $2\leq p<n$. For a different proof of this inequality, see Frank and Loss \cite{FL12}. Obviously, \eqref{HS1} implies from \eqref{HS2}.

{\bf Version on the Heisenberg group}: Let $\mathbb{H}^n$ be the Heisenberg group, that is, the set $\mathbb{R}^{2n+1}$ equipped with the group law 
		$$
		(z,t) \circ (\widetilde{z},\widetilde{t}\,) := (z + \widetilde{z}, t + \widetilde{t}+2\,\text{Im}\langle z,\widetilde{z}\rangle ),
		$$
		where $
		(z, t)=\left(z_{1}, \ldots, z_{n}, t\right)=\left(x_{1}, y_{1}, \ldots, x_{n}, y_{n}, t\right)
		 \in \mathbb{H}^n$, $x\in \mathbb{R}^{n}$, $y\in \mathbb{R}^{n}$, $t\in \mathbb{R}$ and $\mathbb{R}^{2n}$ is identified by $\mathbb{C}^{n}$.
		
		A Hardy inequality in a half-space of the Heisenberg group was shown by Luan and Yang in 2008 \cite{LY08} in the form
		\begin{equation}\label{LY_ineq}
		\int_{\mathbb{H}^{n}_{t>0}}  \frac{|x|^2+|y|^2 }{t^2} |u|^2 d\xi \leq 	\int_{\mathbb{H}^{n}_{t>0}} |\nabla_{H}u|^2 d\xi,
		\end{equation} 
		for every function $u \in C_0^{\infty}(\mathbb{H}^{n}_{t>0})$.
		
		In 2016, Larson generalised the above inequality to any half-space of the Heisenberg group \cite{Larson}: 
		\begin{equation}\label{L2Hardy}
		\frac{1}{4} \int_{\mathbb{H}^+}  \frac{\sum_{i=1}^{n}\langle X_i(\xi), \nu \rangle^2+\langle Y_i(\xi), \nu \rangle^2 }{dist(\xi,\partial \mathbb{H}^+)^2} |u|^2 d\xi \leq  	\int_{\mathbb{H}^+} |\nabla_{H}u|^2 d\xi,
		\end{equation}
		where $X_i$ and $Y_i$ (for $i=1,\ldots,n$) are left-invariant vector fields on the Heisenberg group, $\nu$ is the unit vector. However, the following conjecture remained open (see, e.g. \cite{Larson}).

	{\bf Conjecture:}
			The  $L^p$-version of the above Hardy inequality (in a half-space) should be valid, that is,
			\begin{equation*}
			\left(\frac{p-1}{p}\right)^p \int_{\mathbb{H}^+}  \frac{\mathcal{W}(\xi)^p }{dist(\xi,\partial \mathbb{H}^+)^p} |u|^p d\xi\leq 	\int_{\mathbb{H}^+} |\nabla_{H}u|^p d\xi.
			\end{equation*}
			Also, the constant should be sharp. 

		Here we have the so-called angle function  (see \cite{Garofalo})
		\begin{equation*}
		\mathcal{W}(\xi) := \sqrt{ \sum_{i=1}^{n}\langle X_i(\xi),\nu\rangle^2+\langle Y_i(\xi),\nu\rangle^2}.
		\end{equation*}
The following theorem approves the above conjecture: 
	\begin{thm}\cite{RSS18}
			Let $\mathbb{H}^+$ be a half-space of the Heisenberg group $\mathbb{H}^n$. For $2\leq p<Q$  with $Q=2n+2$, there exists some $C>0$ such that for every function $u \in C_0^{\infty}(\mathbb{H}^+)$ we have 
			\begin{multline}
			C \left( \int_{\mathbb{H}^+} |u|^{p^*} d\xi\right)^{\frac{1}{p^*}} \\ \leq \left( \int_{\mathbb{H}^+} |\nabla_{H} u|^p d\xi - \left(\frac{p-1}{p}\right)^p \int_{\mathbb{H}^+} \frac{ \mathcal{W}(\xi)^p}{dist(\xi,\partial \mathbb{H}^+)^p} |u|^p d\xi \right)^{\frac{1}{p}},
			\end{multline}
			where $p^* := Qp/(Q-p)$ and $dist(\xi,\partial \mathbb{H}^+) := \langle \xi, \nu \rangle - d$.
	\end{thm}

{\bf Horizontal	Poincar\'e inequality.}	The ``horizontal" approach also implies the following Poincar\'e type inequality \cite{Ruzhansky-Suragan:JDE} on stratified groups:
		\begin{equation}\label{Lp-Poincare}
		\frac{|N-p|}{Rp}\left\|f
		\right\|_{L^{p}(\Omega)}\leq\left\|\nabla_{H} f\right\|_{L^{p}(\Omega)},\quad 1<p<\infty,
		\end{equation}
		for $f\in C^{\infty}_{0}(\Omega\backslash\{x'=0\})$ and $R=\underset{x\in\Omega}{\sup}|x'|$.
	
		For example, let us consider the blow-up solutions to the $p$-sub-Laplacian heat equation on the stratified group, that is,
		\begin{align}\label{main_eqn}
		\begin{cases}
		u_t(x,t) - \mathcal{L}_p u(x,t) = f(u(x,t)), \,\,\, & (x,t) \in \Omega \times (0,+\infty), \\ 
		u(x,t)  =0,  \,\,\,& (x,t) \in \partial \Omega \times [0,+\infty), \\
		u(x,0)  = u_0(x)\geq 0,\,\,\, & x \in \overline{\Omega}, 
		\end{cases}
		\end{align}
		where $f$ is locally Lipschitz continuous on $\mathbb{R}$, $f(0)=0$, and such that $f(u)>0$ for $u>0$. Here $\mathcal{L}_p$ is the $p$-sub-Laplacian.
		By using \eqref{Lp-Poincare} it can be proved that 
		nonnegative solution to \eqref{main_eqn} blows up at a finite time $T^*$.	
		Thus, inequality \eqref{Lp-Poincare} is a powerful tool 
		proving the existence or/and nonexistence (blow-up) of the solution of subelliptic partial differential equations. 
		
		However, in general, the constant \eqref{Lp-Poincare} is not optimal. When $p=2$ the optimal constant can be expressed in terms of the positive ground state (if it exists on a stratified group). For general stratified groups the question about positivity of the ground state is open.  
		
		It is important to note that for the first time the horizontal Poincar\'e inequality \eqref{Lp-Poincare} appears in  \cite[Theorem 2.12]{Ambrosio05}  which is valid for general vector fields (including general stratified groups setting).

		Let $\Omega\subset\mathbb{G}$ be an open set and we denote its boundary by $\partial\Omega$. The notation $u\in C^{1}(\Omega)$ means 
		$\nabla_{H} u\in C(\Omega)$. 
			Let $\Omega \subset \mathbb{G}$ be a set supporting the divergence formula on $\mathbb{G}$.
		Let $u\in C_{0}^{1}(\Omega)$ and $0<\phi\in C^{2}(\Omega).$
		We have 
			\begin{equation}
			\left|\nabla_{H} u-\frac{\nabla_{H} \phi}{\phi} u \right|^{2}=
			|\nabla_{H} u|^{2}-\frac{\nabla_{H} \phi}{\phi}  \nabla_{H} u^{2}+\frac{|\nabla_{H} \phi|^{2}}{\phi^{2}}u^{2}		\end{equation}
			and
			\begin{equation}
			-\frac{\nabla_{H} \phi}{\phi}  \nabla_{H} u^{2}
			=-\nabla_{H}\cdot \left(\frac{\nabla_{H} \phi}{\phi} u^{2} \right) +\frac{\mathcal{L} \phi}{\phi}u^{2}-\frac{|\nabla_{H} \phi|^{2}}{\phi^{2}}u^{2}.	\end{equation}
			These imply
			\begin{equation}
			\label{eq33}
			\left|\nabla_{H} u-\frac{\nabla_{H} \phi}{\phi} u \right|^{2}=|\nabla_{H} u|^{2}-\nabla_{H}\cdot \left(\frac{\nabla_{H} \phi}{\phi} u^{2} \right) +\frac{\mathcal{L} \phi}{\phi} u^{2},
			\end{equation}
			that is,
				\begin{equation}
		\int_{\Omega} 	\left|\nabla_{H} u-\frac{\nabla_{H} \phi}{\phi} u \right|^{2}dx=\int_{\Omega}\left(|\nabla_{H} u|^{2}-\nabla_{H}\cdot \left(\frac{\nabla_{H} \phi}{\phi} u^{2} \right) +\frac{\mathcal{L} \phi}{\phi} u^{2}\right) dx.
			\end{equation}	
Now by using the divergence formula (see, e.g. \cite{RS_book}) to the second term in the right hand side, we arrive at

	\begin{equation}\label{ineqG}
0\leq \int_{\Omega} \left|\nabla_{H} u-\frac{\nabla_{H} \phi}{\phi} u \right|^{2}dx= 	\int_{\Omega} \left(|\nabla_{H} u|^{2}+\frac{\mathcal{L} \phi}{\phi} |u|^{2}\right)dx
\end{equation}
for any $u\in C_{0}^{1}(\Omega)$ and $0<\phi\in C^{2}(\Omega).$
			Here	the equality case holds if and only if $u$ is proportional to $\phi$.
			Indeed, we have the equality case if and only if 
			$$0=\left|\nabla_{H} u-\frac{\nabla_{H} \phi}{\phi} u\right|^{2}=\left|\nabla_{H} \left( \frac{u}{\phi}\right)\right|^{2}\phi^{2},$$ 
			that is,
			$X_{k}\left(\frac{u}{\phi}\right)=0,\,\, k=1,...,N.$ Since any left invariant vector field  of $\mathbb{G}$ can be represented by Lie brackets of $\{X_{1},...,X_{N}\}$, we conclude that
			$\frac{u}{\phi}$ is a constant if and only if  $\left|\nabla_{H} \left( \frac{u}{\phi}\right)\right|=0$.

			Consider the following (spectral) problem for the minus Dirichlet sub-Laplacian:
		\begin{align} \label{eigen_1}
		\begin{cases} -\mathcal{L}\phi(\xi)=\mu \phi(\xi), & \xi \in \Omega,\; \; \Omega\subset\mathbb{G}, \\ \phi(\xi)=0, & \xi \in \partial \Omega. \end{cases}
		\end{align}
		
		Let both $\mu> 0$ and $\phi>0$ satisfy \eqref{eigen_1}, that is, $\frac{\mathcal{L} \phi}{\phi}=-\mu$.
		Then \eqref{ineqG} implies the sharp Poincar\'e (or Steklov) inequality on a stratified group $\mathbb{G}$ 
		
		$$
		\int_{\Omega} | u|^{2}dx\leq \frac{1}{\mu} \int_{\Omega} |\nabla_{H}u|^{2}dx.
		$$
		Note that if $\Omega$ is an open smooth bounded set of the Heisenberg group $\mathbb{H}^n$, then there exist $\mu>0$ and $\phi>0$, which are the first eigenvalue and the corresponding eigenfunction of the minus Dirichlet sub-Laplacian, respectively.

		One can iterate the above process to obtain higher order versions of equality \eqref{ineqG}. It is discussed for general real smooth vector fields in \cite{OS19}.
		In turn, this equality follows the proof of the higher order Poincar\'e inequality, 
		characterization of the best constant and its existence as well as characterization of nontrivial extremizers and their existence. Now we restate some results from \cite{OS19} in terms of stratified groups. We also briefly recall their proofs. 
			
		\begin{thm}\label{mainG} \cite{OS19}
			Assume that $\varphi>0$ is a positive eigenfunction of $-\mathcal{L}$ with an eigenvalue $\lambda$, that is, $-\mathcal{L}\varphi= \lambda \varphi$ in $\Omega \subset \mathbb{G}$. For every $u\in C_{0}^{\infty}(\Omega)$
		 the following identities are valid:
			\begin{multline}\label{maineven}
			\left|\nabla_{H}^{2m} u\right|^{2}-\lambda^{2 m} |u|^{2}=\sum_{j=0}^{m-1}\lambda^{2(m-1-j)} \left(\left|\mathcal{L}^{j+1} u+\lambda \mathcal{L}^{j} u\right|^{2}+2 \lambda\left|\nabla_{H} \mathcal{L}^{j} u-\frac{\nabla_{H} \varphi}{\varphi} \mathcal{L}^{j} u\right|^{2}\right)
			\\
			+\sum_{j=0}^{m-1}2 \lambda^{2(m-1-j)+1} \nabla_{H}\cdot\left(\frac{\nabla_{H} \varphi}{\varphi}(\mathcal{L}^{j} u)^{2}-\mathcal{L}^{j}u\nabla_{H}\mathcal{L}^{j}u\right),
			\end{multline}
			
			where $m=1,2,\ldots,$ and
			
			\begin{multline}\label{mainodd}
			\left|\nabla_{H}^{2m+1} u\right|^{2}-\lambda^{2 m+1} |u|^{2}=\left|\nabla_{H} \mathcal{L}^{m} u-\frac{\nabla_{H} \varphi}{\varphi} \mathcal{L}^{m} u\right|^{2}
			\\
			+\sum_{j=0}^{m-1} \lambda^{2(m-j)-1}\left(\left|\mathcal{L}^{j+1} u+\lambda \mathcal{L}^{j} u\right|^{2}+2 \lambda\left|\nabla_{H} \mathcal{L}^{j} u-\frac{\nabla_{H} \varphi}{\varphi} \mathcal{L}^{j} u\right|^{2}\right)
			\\
			+2 \sum_{j=0}^{m-1} \lambda^{2(m-j)} \nabla_{H}\cdot\left(\frac{\nabla_{H} \varphi}{\varphi}\left(\mathcal{L}^{j} u\right)^{2}-\mathcal{L}^{j}u\nabla_{H}\mathcal{L}^{j}u\right)+\nabla_{H}\cdot\left(\frac{\nabla_{H} \varphi}{\varphi}\left(\mathcal{L}^{m} u\right)^{2}\right),
			\end{multline}
			where $m=0,1,2,\ldots.$
		\end{thm}

	Theorem \ref{mainG} has the following interesting consequence in the Euclidean setting.
		
		\begin{thm} \label{main2G} \cite{OS19}
			Let $\Omega \subset \mathbb{R}^{n}$ be a connected domain, for which the divergence theorem is true.
			Then we have the remainder of the higher order Poincar\'e inequality 
			
			\begin{multline}\label{mainevenLap}
			\int_{\Omega}\left|\nabla^{2m} u\right|^{2}dx-\lambda_{1}^{2 m} \int_{\Omega}|u|^{2}dx\\=\sum_{j=0}^{m-1}\lambda_{1}^{2(m-1-j)} \left(\int_{\Omega}\left|\Delta^{j+1} u+\lambda_{1} \Delta^{j} u\right|^{2}dx+2 \lambda_{1}\int_{\Omega}\left|\nabla \Delta^{j} u-\frac{\nabla u_{1}}{u_{1}} \Delta^{j} u\right|^{2}dx\right)\geq 0,
			\end{multline}
			where $m=1,2,\dots,$ and
			\begin{multline}\label{mainoddLap}
			\int_{\Omega}\left|\nabla^{2m+1} u\right|^{2}dx-\lambda_{1}^{2 m+1} 	\int_{\Omega}|u|^{2}dx=	\int_{\Omega}\left|\nabla \Delta^{m} u-\frac{\nabla u_{1}}{u_{1}} \Delta^{m} u\right|^{2}dx
			\\
			+\sum_{j=0}^{m-1} \lambda_{1}^{2(m-j)-1}\left(	\int_{\Omega}\left|\Delta^{j+1} u+\lambda_{1} \Delta^{j} u\right|^{2}dx+2 \lambda_{1}	\int_{\Omega}\left|\nabla \Delta^{j} u-\frac{\nabla u_{1}}{u_{1}} \Delta^{j} u\right|^{2}dx\right)\geq 0,
			\end{multline}
			where $m=0,1,\dots,$ for all $u\in C_{0}^{\infty}(\Omega).$ Here $u_{1}$ is the ground state of the Dirichlet Laplacian $-\Delta$ in $\Omega$ and $\lambda_{1}$ is the corresponding eigenvalue. The equality cases hold if and only if $u$ is proportional to $u_{1}$.
		\end{thm}

		\begin{proof}[Proof of Theorem \ref{mainG}]
			
			For $m=1$, a direct computation yields 
			
			\begin{multline}\label{m=2}
			\left|\mathcal{L} u-\frac{\mathcal{L} \varphi}{\varphi} u\right|^{2}=|\mathcal{L} u|^{2}-2 \frac{\mathcal{L} \varphi}{\varphi} u \mathcal{L} u+\left(\frac{\mathcal{L} \varphi}{\varphi}\right)^{2} |u|^{2}
			\\
			=|\mathcal{L} u|^{2}-\frac{\mathcal{L} \varphi}{\varphi}\left(\mathcal{L}\left(|u|^{2}\right)-2|\nabla_{H} u|^{2}\right)+\left(\frac{\mathcal{L} \varphi}{\varphi}\right)^{2} |u|^{2},
			\end{multline}
			$$
			-\frac{\mathcal{L} \varphi}{\varphi} \mathcal{L}\left(|u|^{2}\right)=2\lambda \nabla_{H}\cdot (u\nabla_{H}u)
			$$
			and
			$$
			2 \frac{\mathcal{L} \varphi}{\varphi}|\nabla_{H} u|^{2}=2 \frac{\mathcal{L} \varphi}{\varphi}\left(-\frac{\mathcal{L} \varphi}{\varphi} |u|^{2}+\left|\nabla_{H} u-\frac{\nabla_{H} \varphi}{\varphi} u\right|^{2}+\nabla_{H}\cdot\left(\frac{\nabla_{H} \varphi}{\varphi} |u|^{2}\right)\right).
			$$
		With $-\mathcal{L}\varphi= \lambda \varphi$ these follow that 
			$$
			\left|\mathcal{L} u-\frac{\mathcal{L} \varphi}{\varphi} u\right|^{2}	=|\mathcal{L} u|^{2}-\left(\frac{\mathcal{L} \varphi}{\varphi}\right)^{2} |u|^{2}
			$$
			
			$$
		+2 \frac{\mathcal{L} \varphi}{\varphi}\left(\left|\nabla_{H} u-\frac{\nabla_{H} \varphi}{\varphi} u\right|^{2}+\nabla_{H}\cdot\left(\frac{\nabla_{H} \varphi}{\varphi} |u|^{2}\right)\right)+2\lambda \nabla_{H}\cdot (u\nabla_{H}u)
			$$

			$$
			=|\mathcal{L} u|^{2}-\lambda^{2} |u|^{2}-2 \lambda\left(\left|\nabla_{H} u-\frac{\nabla_{H} \varphi}{\varphi} u\right|^{2}+\nabla_{H}\cdot\left(\frac{\nabla_{H} \varphi}{\varphi} |u|^{2}\right)\right)+2\lambda \nabla_{H}\cdot (u\nabla_{H}u),
			$$
			that is,
			$$|\mathcal{L} u|^{2}-\lambda^{2} |u|^{2}=\left|\mathcal{L} u+\lambda u\right|^{2}+2 \lambda\left(\left|\nabla_{H} u-\frac{\nabla_{H} \varphi}{\varphi} u\right|^{2}+\nabla_{H}\cdot\left(\frac{\nabla_{H} \varphi}{\varphi} |u|^{2}-u\nabla_{H}u\right)\right).$$
			By the induction
			$m \Rightarrow m+1,$ we establish
			
			$$
			\left|\mathcal{L}^{m +1} u\right|^{2}-\lambda^{2(m+1)} |u|^{2}
			$$
			
			$$
			=\left|\mathcal{L} \mathcal{L}^{m} u\right|^{2}-\lambda^{2}\left|\mathcal{L}^{m} u\right|^{2}+\lambda^{2}\left(\left|\mathcal{L}^{m} u\right|^{2}-\lambda^{2 m} |u|^{2}\right)
			$$
			
			$$
			=\left|\mathcal{L}^{m+1} u+\lambda \mathcal{L}^{m} u\right|^{2}+2 \lambda\left|\nabla_{H} \mathcal{L}^{m} u-\frac{\nabla_{H} \varphi}{\varphi} \mathcal{L}^{m} u\right|^{2}
			$$
			
			$$
			+2 \lambda \nabla_{H}\cdot\left(\frac{\nabla_{H} \varphi}{\varphi}\left(\mathcal{L}^{m} u\right)^{2}-\mathcal{L}^{m}u\nabla_{H}\mathcal{L}^{m}u\right)
			$$
			
			$$
			+\sum_{j=0}^{m-1} \lambda^{2(m-j)}\left(\left|\mathcal{L}^{j+1} u+\lambda \mathcal{L}^{j} u\right|^{2}+2 \lambda\left|\nabla_{H} \mathcal{L}^{j} u-\frac{\nabla_{H} \varphi}{\varphi} \mathcal{L}^{j} u\right|^{2}\right)
			$$
			
			$$
			+2 \sum_{j=0}^{m-1} \lambda^{2(m-j)} \left(\lambda \nabla_{H}\cdot\left(\frac{\nabla_{H} \varphi}{\varphi}\left(\mathcal{L}^{j} u\right)^{2}-\mathcal{L}^{j}u\nabla_{H}\mathcal{L}^{j}u\right)\right)
			$$
			
			$$
			=\sum_{j=0}^{m} \lambda^{2(m-j)}\left(\left|\mathcal{L}^{j+1} u+\lambda \mathcal{L}^{j} u\right|^{2}+2 \lambda\left|\nabla_{H} \mathcal{L}^{j} u-\frac{\nabla_{H} \varphi}{\varphi} \mathcal{L}^{j} u\right|^{2}\right)
			$$
			
			$$
			+2 \sum_{j=0}^{m} \lambda^{2(m-j)+1}  \nabla_{H}\cdot\left(\frac{\nabla_{H} \varphi}{\varphi}\left(\mathcal{L}^{j} u\right)^{2}-\mathcal{L}^{j}u\nabla_{H}\mathcal{L}^{j}u\right).
			$$
			It proves formula \eqref{maineven}.

			Now it remains to show relation \eqref{mainodd}. We have			\begin{multline}
			\label{eq1}
			\left|\nabla_{H} u-\frac{\nabla_{H} \varphi}{\varphi} u \right|^{2}=|\nabla_{H} u|^{2}-2\frac{\nabla_{H} \varphi}{\varphi} u \nabla_{H} u+\frac{|\nabla_{H} \varphi|^{2}}{\varphi^{2}}|u|^{2}\\=
			|\nabla_{H} u|^{2}-\frac{\nabla_{H} \varphi}{\varphi}  \nabla_{H} |u|^{2}+\frac{|\nabla_{H} \varphi|^{2}}{\varphi^{2}}|u|^{2},		\end{multline}
		and
			\begin{equation}
			\label{eq2} -\frac{\nabla_{H} \varphi}{\varphi}  \nabla_{H} |u|^{2}
			=-\nabla_{H}\cdot \left(\frac{\nabla_{H} \varphi}{\varphi} |u|^{2} \right) +\frac{\mathcal{L} \varphi}{\varphi}|u|^{2}-\frac{|\nabla_{H} \varphi|^{2}}{\varphi^{2}}|u|^{2}.	\end{equation}
			Equalities \eqref{eq1} and \eqref{eq2} imply 
			\begin{equation}
			\label{eq3}
			|\nabla_{H} u|^{2}-\lambda |u|^{2}=\left|\nabla_{H} u-\frac{\nabla_{H} \varphi}{\varphi} u\right|^{2}+\nabla_{H}\cdot\left(\frac{\nabla_{H} \varphi}{\varphi} |u|^{2}\right).
			\end{equation}
			It gives
			\eqref{mainodd} when $m=0.$
			
			Now by using the scaling 
			$$
			u \mapsto \mathcal{L}^{m} u
			$$
			to \eqref{eq3},  we obtain 
			\begin{equation}
			\label{eq4}
			\left|\nabla_{H} \mathcal{L}^{m} u\right|^{2}=\left|\nabla_{H} \mathcal{L}^{m} u-\frac{\nabla_{H} \varphi}{\varphi} \mathcal{L}^{m} u\right|^{2}+\nabla_{H}\cdot\left(\frac{\nabla_{H} \varphi}{\varphi}\left(\mathcal{L}^{m} u\right)^{2}\right)+\lambda\left(\mathcal{L}^{m} u\right)^{2}.
			\end{equation}
			
			Finally, by using formula \eqref{maineven}, we establish
			$$
			\left|\nabla_{H} \mathcal{L}^{m} u\right|^{2}-\lambda^{2 m+1} |u|^{2}
			$$
			
			$$
			=\left|\nabla_{H} \mathcal{L}^{m} u-\frac{\nabla_{H} \varphi}{\varphi} \mathcal{L}^{m} u\right|^{2}+\nabla_{H}\cdot\left(\frac{\nabla_{H} \varphi}{\varphi}\left(\mathcal{L}^{m} u\right)^{2}\right)
			$$
			
			$$
			+\lambda\left(\left(\mathcal{L}^{m} u\right)^{2}-\lambda^{2 m} |u|^{2}\right)
			$$
			
			$$
			=\left|\nabla_{H} \mathcal{L}^{m} u-\frac{\nabla_{H} \varphi}{\varphi} \mathcal{L}^{m} u\right|^{2}+\nabla_{H}\cdot\left(\frac{\nabla_{H} \varphi}{\varphi}\left(\mathcal{L}^{m} u\right)^{2}\right)
			$$
			
			$$
			+\sum_{j=0}^{m-1} \lambda^{2(m-j)-1}\left(\left|\mathcal{L}^{j+1} u+\lambda \mathcal{L}^{j} u\right|^{2}+2 \lambda\left|\nabla_{H} \mathcal{L}^{j} u-\frac{\nabla_{H} \varphi}{\varphi} \mathcal{L}^{j} u\right|^{2}\right)
			$$
			
			$$
			+2 \sum_{j=0}^{m-1} \lambda^{2(m-j)}  \nabla_{H}\cdot\left(\frac{\nabla_{H} \varphi}{\varphi}\left(\mathcal{L}^{j} u\right)^{2}\right).
			$$

		\end{proof}

	{Picone type representation formula}: 
		\begin{thm}\cite{OS19}\label{mainL2m}  
			For all $u\in C^{1}(\Omega)$ and $\varphi\in C^{2}(\Omega)$ with $\varphi>0$, we have 
			\begin{multline}\label{poincareLpm}
			|\nabla_{H} u|^{p_{m}}+\left(\frac{\mathcal{L} \varphi}{\varphi}+\sigma_{m}\right) u^{p_{m}}
			=\sum_{j=1}^{m-1}\left|| \nabla_{H}\left(u^{p_{m-j-1}}\right)|^{p_{j}}-2^{p_{j}-1} u^{p_{m-1}}\right|^{2}
			\\
			+\left|\nabla_{H}\left(u^{p_{m-1}}\right)-\frac{\nabla_{H} \varphi}{\varphi} u^{p_{m-1}}\right|^{2}+\nabla_{H}\cdot\left(\frac{\nabla_{H} \varphi}{\varphi} u^{p_{m}}\right) \text{in}\; \Omega\subset \mathbb{G},
			\end{multline}
			where where $m$ is a nonnegative integer.
			Here $p_{m}=2^{m},\, m \geqslant 0,$
			and	
			$\sigma_{m}=\frac{1}{4} \sum_{j=1}^{m-1} 4^{p_{j}},\, m \geqslant 1.$
		\end{thm}

			Theorem \ref{mainL2m} has the following interesting consequence in the Euclidean setting.
		\begin{thm}\cite{OS19}\label{mainLp2}
			Let $\Omega \subset \mathbb{R}^{n}$ be a connected domain, for which the divergence theorem is true. 
			For all $u\in C_{0}^{1}(\Omega)$, we have  
			
			\begin{multline}\label{mainLpLap}
			\int_{\Omega}|\nabla u|^{p_{m}}dx-\left(\lambda_{1}-\sigma_{m}\right) \int_{\Omega}|u|^{p_{m}}dx\\=\sum_{j=1}^{m-1} \int_{\Omega}\left|| \nabla\left(u^{p_{m-j-1}}\right)|^{p_{j}}-2^{p_{j}-1} u^{p_{m-1}}\right|^{2}dx+\int_{\Omega}\left|\nabla\left(u^{p_{m-1}}\right)-\frac{\nabla u_{1}}{u_{1}} u^{p_{m-1}}\right|^{2}dx\geq 0,
			\end{multline}
			 where $\sigma_{m}=\frac{1}{4} \sum_{j=1}^{m-1} 4^{p_{j}},\, m \in \mathbb{N},\, p_{j}=2^{j},$ $u_{1}$ is the ground state of the minus Dirichlet Laplacian $-\Delta$ in $\Omega$ and $\lambda_{1}$ is the corresponding eigenvalue. 
		\end{thm}
		
		Note that, for $m=1$ the sigma notation term in \eqref{mainLpLap} disappears as usual (since the lower index is greater than the upper one). 
		It is important to observe that \eqref{mainLpLap}  can be considered as a remainder term for some $L^{p}$-Poincar\'e inequalities (which are also commonly called as $L^{p}$-Friedrichs inequalities). In  general, determining the sharp constant in the $L^{p}$-Friedrichs inequality is an open problem.

		\begin{proof}[Proof of Theorem \ref{mainL2m}]
			When $m=1$, we have $p_{1}=2$, $\sigma_{1}=0$,
			and
			
			$$
			|\nabla_{H} u|^{2}+\frac{\mathcal{L} \varphi}{\varphi} u^{2}=\left|\nabla_{H} u-\frac{\nabla_{H} \varphi}{\varphi} u\right|^{2}+\nabla_{H}\cdot\left(\frac{\nabla_{H} \varphi}{\varphi} u\right).
			$$
			When $m=2$, we have $p_{2}=4$, $\sigma_{2}=\frac{1}{4}\, 4^{p_{1}}=4$, and
			$$
			|\nabla_{H} u|^{4}+\left(\frac{\mathcal{L} \varphi}{\varphi}+4\right) u^{4}=\left.|| \nabla_{H} u\right|^{2}-\left.2 |u|^{2}\right|^{2}+
			$$
			$$
			+\left|\nabla_{H}\left(|u|^{2}\right)-\frac{\nabla_{H} \varphi}{\varphi} |u|^{2}\right|^{2}+\nabla_{H}\cdot\left(\frac{\nabla_{H} \varphi}{\varphi} u^{4}\right).
			$$
			In order to use the induction process, we observe  
			
			$$\left|| \nabla_{H} u |^{p_{m}}-2^{p_{m}-1} u^{p_{m}}\right|^{2}
			=|\nabla_{H} u|^{p_{m+1}}-2^{p_{m}} u^{p_{m}}|\nabla_{H} u|^{p_{m}}+\frac{1}{4}\, 4^{p_{m}} u^{p_{m+1}}.
			$$
			Plugging in \(|u|^{2}\) instead of \(u\)  in \eqref{poincareLpm} we get 
			
			$$
			\left|\nabla_{H}\left(|u|^{2}\right)\right|^{p_{m}}+\left(\frac{\mathcal{L} \varphi}{\varphi}+\sigma_{m}\right)\left(|u|^{2}\right)^{p_{m}}
			$$
			
			$$
			=\sum_{j=1}^{m-1} \left|\left|\nabla_{H}\left(|u|^{2}\right)^{p_{m-j-1}}\right|^{p_{j}}-2^{p_{j}-1}\left(|u|^{2}\right)^{p_{m-1}}\right|^{2}
			$$
			
			$$
			+\left|\nabla_{H}\left(|u|^{2}\right)^{p_{m-1}} -\frac{\nabla_{H} \varphi}{\varphi}\left(|u|^{2}\right)^{p_{m-1}}\right|^{2}+\nabla_{H}\cdot\left(\frac{\nabla_{H} \varphi}{\varphi}\left(|u|^{2}\right)^{p_{m}}\right),
			$$
			where 
			
			$$
			\left|\nabla_{H}\left(|u|^{2}\right)\right|^{p_{m}}+\left(\frac{\mathcal{L} \varphi}{\varphi}+\sigma_{m}\right)\left(|u|^{2}\right)^{p_{m}}=2^{p_{m}} u^{p_{m}}\left|\nabla_{H} u \right|^{p_{m}}+\left(\frac{\mathcal{L} \varphi}{\varphi}+\sigma_{m}\right) u^{p_{m+1}}
			$$
			and 
			$$
			\sum_{j=1}^{m-1} \left|\left|\nabla_{H}\left(|u|^{2}\right)^{p_{m-j-1}}\right|^{p_{j}}-2^{p_{j}-1}\left(|u|^{2}\right)^{p_{m-1}}\right|^{2}
			$$
			
			$$
			+\left|\nabla_{H}\left(|u|^{2}\right)^{p_{m-1}} -\frac{\nabla_{H} \varphi}{\varphi}\left(|u|^{2}\right)^{p_{m-1}}\right|^{2}+\nabla_{H}\cdot\left(\frac{\nabla_{H} \varphi}{\varphi}\left(|u|^{2}\right)^{p_{m}}\right)
			$$
			$$
			=\sum_{j=1}^{m-1} \left| |\nabla_{H}\left(u^{p_{m-j}}\right)|^{p_{j}}-2^{p_{j}-1} u^{p_{m}}\right|^{2}
			$$
			
			$$
			+\left|\nabla_{H}\left(u^{p_{m}}\right)-\frac{\nabla_{H} \varphi}{\varphi} u^{p_{m}}\right|^{2}+\nabla_{H}\cdot\left(\frac{\nabla_{H} \varphi}{\varphi} u^{p_{m+1}}\right).
			$$
		It yields 
			$$|\nabla_{H} u|^{p_{m+1}}
			=2^{p_{m}} u^{p_{m}}\left|\nabla_{H} u\right|^{p_{m}}+\left|| \nabla_{H} u|^{p_{m}}-2^{p_{m}-1} u^{p_{m}}\right|^{2}-\frac{1}{4}\, 4^{p_{m}} u^{p_{m+1}}
			$$
			$$
			=-\left(\frac{\mathcal{L} \varphi}{\varphi}+\sigma_{m}+\frac{1}{4}\, 4^{p_{m}}\right) u^{p_{m+1}}+\left| |\nabla_{H} u|^{p_{m}}-2^{p_{m}-1} u^{p_{m}}\right|^{2}
			$$
			$$
			+\sum_{j=1}^{m-1} \left| |\nabla_{H}\left(u^{p_{m-j}}\right)|^{p_{j}}-2^{p_{j}-1} u^{p_{m}}\right|^{2}
			$$
			
			$$
			+\left|\nabla_{H}\left(u^{p_{m}}\right)-\frac{\nabla_{H} \varphi}{\varphi} u^{p_{m}}\right|^{2}+\nabla_{H}\cdot\left(\frac{\nabla_{H} \varphi}{\varphi} u^{p_{m+1}}\right)
			$$
			$$
				=-\left(\frac{\mathcal{L} \varphi}{\varphi}+\sigma_{m+1}\right) u^{p_{m+1}}+\left| |\nabla_{H} u|^{p_{m}}-2^{p_{m}-1} u^{p_{m}}\right|^{2}
			$$
			$$
			+\sum_{j=1}^{m-1} \left| |\nabla_{H}\left(u^{p_{m-j}}\right)|^{p_{j}}-2^{p_{j}-1} u^{p_{m}}\right|^{2}
			$$
			
			$$
			+\left|\nabla_{H}\left(u^{p_{m}}\right)-\frac{\nabla_{H} \varphi}{\varphi} u^{p_{m}}\right|^{2}+\nabla_{H}\cdot\left(\frac{\nabla_{H} \varphi}{\varphi} u^{p_{m+1}}\right).
			$$
		\end{proof}
	
	Finally, note that in addition to the above discussed two approaches (the fundamental solution and horizontal estimate approach) there is another interesting approach for Hardy-type inequalities on Heisenberg group so called the \emph{CC-distance approach}.  For discussions in this direction we suggest  \cite{BCG06}, \cite{BCX07}, \cite{CDPT}, \cite{FranPra19}, \cite{Lehrb} and \cite{Yang13} as well as references therein. 
	
\section{Hardy type inequalities on homogeneous groups}
	\label{Sec3}
	By the definition, there is no homogeneous (horizontal) gradient on {non-stratified} graded groups, so there is no horizontal estimates.
A non-stratified graded group $\mathbb{G}$ may not have a homogeneous sub-Laplacian, but it always has so-called  Rockland operators, which are left-invariant homogeneous subelliptic differential operators on $\mathbb{G}$. Therefore, the fundamental solution approach can be applied to general graded groups.

{\bf Beyond graded groups}: There is no invariant homogeneous subelliptic differential operator on non-graded homogeneous groups; in particular, {no fundamental solution}.
	
		Question is even: 
		\begin{itemize}
			\item How to formulate Hardy's inequality on non-graded homogeneous groups?
		\end{itemize} 
		
		A systematic analysis towards an answer to this question was presented recently in the book form \cite{RS_book}. One of the key ideas was consistently working with the quasi-radial derivative operator $\mathcal{R}_{|x|}:=\frac{d}{d|x|}$  	and with the  Euler operator $\mathbb{E}:=|x|\frac{d}{d|x|}$
		to obtain homogeneous group analogues of the Hardy type inequalities.

		Actually, it can be shown that any (connected, simply connected) nilpotent Lie group is some $\mathbb{R}^{n}$ with a { polynomial group law}: $\mathbb{R}^{n}$ with linear group law, $\mathbb H^{n}$ with quadratic group law, etc. So we can identify $\mathbb{G}$ with $\mathbb{R}^{n}$ (topologically).
		
		\begin{defn}If a Lie group (on $\mathbb{R}^{n}$) $\mathbb{G}$ has a property that there exist $n$-real numbers
		$\nu_1,...,\nu_n$ such that the dilation
		$$D_{\lambda}(x):=(\lambda^{\nu_1}x_1,...,\lambda^{\nu_n}x_n), \quad D_{\lambda}:\mathbb{R}^{n}\rightarrow\mathbb{R}^{n},$$
		is an automorphism of the group $\mathbb{G}$ for each $\lambda>0$, then it is called a
		homogeneous group.
			\end{defn}
				
		Let us fix a basis $\{X_{1},\ldots,X_{n}\}$ of the Lie algebra $\mathfrak{g}$
		of the homogeneous group $\mathbb{G}$ such that $X_{k}$ is homogeneous of degree $\nu_{k}$.
		
		Then the homogeneous dimension of $\mathbb{G}$ is
		$$
		Q=\nu_{1}+\cdots+\nu_{n}.
		$$

		A class of homogeneous groups is one of most general subclasses of
		nilpotent Lie groups, that is, the class of homogeneous groups gives
		almost the class of all nilpotent Lie groups but is not equal to it. In 1970, Dyer
		gave an example of a (nine-dimensional) nilpotent Lie group that does
		not allow for any family of dilations \cite{Dyer}.

		Special cases of the homogeneous groups:
		\begin{itemize}
			\item the Euclidean group ($\mathbb{R}^{n}; +$),
			\item H-type groups,
			\item stratified groups,
			\item graded (Lie) groups.
		\end{itemize}
	
	 We also recall the standard Lebesgue measure $dx$ on $\mathbb{R}^{n}$ is the Haar measure for $\mathbb{G}$. It makes 
	the class of homogeneous groups convenient for  analysis. One also can assume that the origin $0$ of $\mathbb{R}^{n}$ is the identity of $\mathbb{G}$, If it is not, then by using a smooth diffeomorphism one can obtain a new (isomorphic) homogeneous group which has the identity $0$. For further discussions in this direction we refer to a recent open access book \cite{RS_book}.

		Let $\mathbb{G}$ be a homogeneous group
		of homogeneous dimension $Q$. Then for all $f\in C_{0}^{\infty}(\mathbb{G}\backslash\{0\})$ and  for any homogeneous quasi-norm $|\cdot|$, we have the following
	
		Hardy inequality on homogeneous groups \cite{RS17Kac}:
		$$
		\left\|\frac{f}{|x|}\right\|_{L^{p}(\mathbb{G})}\leq\frac{p}{Q-p}\left\|\mathcal{R}_{|x|} f\right\|_{L^{p}(\mathbb{G})}, \quad 1<p<Q,
		$$
		for all $f\in C^{\infty}_{0}(\mathbb{G}\backslash\{0\})$, where $\mathcal{R}_{|x|}:=\frac{d}{d|x|}$. Moreover, the constant above is sharp and is attained if and only if $f=0$.

	The main idea to prove these type of inequalities is consistently to work with radial derivative $\mathcal{R}_{|x|}:=\frac{d}{d|x|}$
	and with the  Euler operator $\mathbb{E}:=|x|\frac{d}{d|x|},$ so to obtain relations on homogeneous groups in terms of $\mathcal{R}_{|x|}$ or/and $\mathbb{E}$.
		These yield many inequalities: Hardy, Rellich, Caffarelli-Kohn-Nirenberg, Sobolev type, ..., with best constants for { any homogeneous quasi-norm}.

		\begin{thm}\cite{RSCommunContempMath}\label{aHardy}
			Let $\mathbb{G}$ be a homogeneous group
			of homogeneous dimension $Q\geq 3$.
			Then for every complex-valued function $f\in C^{\infty}_{0}(\mathbb{G}\backslash\{0\})$
			and any homogeneous quasi-norm $|\cdot|$ on $\mathbb{G}$ we have
			\begin{multline}\label{awH}
			\left\|\frac{1}{|x|^{\alpha}}\mathcal{R} f\right\|^{2}_{L^{2}(\mathbb{G})}-\left(\frac{Q-2}{2}-\alpha\right)^{2}
			\left\|\frac{f}{|x|^{\alpha+1}}\right\|^{2}_{L^{2}(\mathbb{G})}
			\\=\left\|\frac{1}{|x|^{\alpha}}\mathcal{R} f+\frac{Q-2-2\alpha}{2|x|^{\alpha+1}}f
			\right\|^{2}_{L^{2}(\mathbb{G})}
			\end{multline}
			for all $\alpha\in\mathbb{R}.$
		\end{thm}

		As a consequence of \eqref{awH}, we obtain the weighted Hardy
		inequality on the homogeneous group $\mathbb{G},$ with the sharp constant:

		For all complex-valued functions $f\in C^{\infty}_{0}(\mathbb{G}\backslash\{0\})$ we have
		\begin{equation}\label{awHardyeq-g}
		\frac{|Q-2-2\alpha|}{2}\left\|\frac{f}{|x|^{\alpha+1}}\right\|_{L^{2}(\mathbb{G})}\leq
		\left\|\frac{1}{|x|^{\alpha}}\mathcal{R} f\right\|_{L^{2}(\mathbb{G})},\quad \forall\alpha\in\mathbb R.
		\end{equation}
		If $\alpha\neq\frac{Q-2}{2}$, then constant in \eqref{awHardyeq-g} is sharp for any homogeneous quasi-norm $|\cdot|$ on $\mathbb{G}$, and inequality \eqref{awHardyeq-g} is attained if and only if $f=0$.

		In the Euclidean case ${\mathbb G}=(\mathbb R^{n},+)$, $n\geq 3$, we have
		$Q=n$, so for any homogeneous quasi-norm $|\cdot|$ on $\mathbb R^{n}$, \eqref{awHardyeq-g} implies
		a new inequality with the optimal constant:
		\begin{equation}\label{Hardy1-r}
		\frac{|n-2-2\alpha|}{2}\left\|\frac{f}{|x|^{\alpha+1}}\right\|_{L^{2}(\mathbb{R}^{n})}\leq
		\left\|\frac{1}{|x|^{\alpha}}\frac{df}{d|x|}\right\|_{L^{2}(\mathbb{R}^{n})},
		\end{equation}
		for all $\alpha\in\mathbb{R}$. We observe that this inequality holds for any homogeneous quasi-norm on $\mathbb R^{n}$.

		\begin{center}{ Note that the constant in \eqref{Hardy1-r} is optimal for any homogeneous quasi-norm.} \end{center}

		For the standard Euclidean distance $\|x\|=\sqrt{x^{2}_{1}+\ldots+x^{2}_{n}}$, by using Schwarz's inequality, this implies  (with the optimal constant):
		\begin{equation}\label{CKN}
		\frac{|n-2-2\alpha|}{2}\left\|\frac{f}{\|x\|^{\alpha+1}}\right\|_{L^{2}(\mathbb{R}^{n})}\leq
		\left\|\frac{1}{\|x\|^{\alpha}}\nabla f\right\|_{L^{2}(\mathbb{R}^{n})},
		\end{equation}
		for all $f\in C_{0}^{\infty}(\mathbb{R}^{n}\backslash\{0\}).$
		
		With ${ \alpha=0}$ we have 
		
		$$
		\left\|\frac{f}{ \|x\|}\right\|_{L^{2}(\mathbb{R}^{n})}\leq  \frac{2}{n-2}\|\nabla f\|_{L^{2}(\mathbb{R}^{n})}, \;\;n\geq 3.
		$$

Moreover, these also can be extended to the weighted $L^p$-Hardy inequalities \cite{RSY18_Tran}.
		
		Let $\mathbb{G}$ be a homogeneous group
		of homogeneous dimension $Q$ and let $\alpha\in \mathbb{R}$.
		Then for all complex-valued functions $f\in C^{\infty}_{0}(\mathbb{G}\backslash\{0\}),$ and any homogeneous quasi-norm $|\cdot|$ on $\mathbb{G}$ for $\alpha p \neq Q$ we have
		$$
		\left\|\frac{f}{|x|^{\alpha}}\right\|_{L^{p}(\mathbb{G})}\leq
		\left|\frac{p}{Q-\alpha p}\right|\left\|\frac{1}{|x|^{\alpha}}\mathbb{E} f\right\|_{L^{p}(\mathbb{G})},\; 1<p<\infty.
		$$
		If $\alpha p\neq Q$ then the constant $\left|\frac{p}{Q-\alpha p}\right|$ is sharp.
		
		For $\alpha p=Q$ we have {the critical case} (believe to be already new in $\mathbb R^n$)
		$$
		\left\|\frac{f}{|x|^{\frac{Q}{p}}}\right\|_{L^{p}(\mathbb{G})}\leq
		p\left\|\frac{\log|x|}{|x|^{\frac{Q}{p}}}\mathbb{E} f\right\|_{L^{p}(\mathbb{G})},
		$$
		where the constant $p$ is sharp.

{\bf Rellich inequalities on homogeneous groups}:
			Let $\mathbb{G}$ be a homogeneous group
			of homogeneous dimension $Q\geq 5.$
			Let $|\cdot|$ be any homogeneous quasi-norm on $\mathbb{G}$.
			Then for every $f\in C^{\infty}_{0}(\mathbb{G}\backslash\{0\})$ \cite{RS17}:
			{ \begin{multline*}
				\left\|\mathcal{R} ^{2} f+\frac{Q-1}{|x|}\mathcal{R}  f
				+\frac{Q(Q-4)}{4|x|^{2}}f\right\|^{2}_{L^{2}(\mathbb{G})}
				+\frac{Q(Q-4)}{2}\left\|\frac{1}{|x|}\mathcal{R}  f
				+\frac{Q-4}{2|x|^{2}}f\right\|^{2}_{L^{2}(\mathbb{G})}\\
				=\left\|\mathcal{R} ^{2} f+\frac{Q-1}{|x|}\mathcal{R}  f
				\right\|^{2}_{L^{2}(\mathbb{G})}-\left(\frac{Q(Q-4)}{4}\right)^{2}\left\|\frac{f}{|x|^{2}}
				\right\|^{2}_{L^{2}(\mathbb{G})},
				\end{multline*}
			}
which implies the (quasi-radial) Rellich inequality
			$$
			\left\|\frac{f}{|x|^{2}}
			\right\|_{L^{2}(\mathbb{G})}\leq\frac{4}{Q(Q-4)}\left\|\mathcal{R} ^{2} f+\frac{Q-1}{|x|}\mathcal{R}  f
			\right\|_{L^{2}(\mathbb{G})},\quad Q\geq 5.
			$$
			The constant $\frac{4}{Q(Q-4)}$ is sharp and it is attained if and only if $f=0$.

After our paper \cite{RS17} the Rellich inequality was  extended (see, \cite{NLNJMSLJ}  and \cite{NLNArchiv})
		to the range $1<p<Q/2$: 
		$$
		\left\|\frac{f}{|x|^{2}}\right\|_{L^{p}(\mathbb{G})}\leq\frac{p^{2}}{Q(p-1)(Q-2p)}\left\|\mathcal{R}_{|x|}^{2}f+\frac{Q-1}{|x|}\mathcal{R}_{|x|}f\right\|
		_{L^{p}(\mathbb{G})},
		$$
		for all $f\in C_{0}^{\infty}(\mathbb{G}\backslash\{0\})$.
		The constant is sharp and it is attained if and only if $f=0$.

One can also obtain the Sobolev type inequalities through identities:
		\begin{thm} \cite{RSY:18}
			Let $\mathbb{G}$ be a homogeneous group
			of homogeneous dimension $Q$. Then for all $f\in C_{0}^{\infty}(\mathbb{G}\backslash\{0\})$ and $1<p<Q$
			\begin{equation}\label{Sob_id}
			\qquad
			\left\|\frac{p}{Q}\mathbb{E} f\right\|^{p}_{L^{p}(\mathbb{G})}
			-\left\|f\right\|^{p}_{L^{p}(\mathbb{G})}=p\int_{\mathbb{G}}I_{p}\left(f,-\frac{p}{Q}\mathbb{E}f\right)\left|f+\frac{p}{Q}\mathbb{E}f\right|^{2}dx, \end{equation}
			where $\mathbb{E}:=|x|\frac{d}{d|x|}$ is the Euler operator, and $I_{p}$ is given by
			$$
			I_{p}(h,g)=(p-1)\int_{0}^{1}|\xi h+(1-\xi)g|^{p-2}\xi d\xi.
			$$
		\end{thm}

		The identity \eqref{Sob_id} implies, for all $f\in C_{0}^{\infty}(\mathbb{G}\backslash\{0\})$, the $L^{p}$-Sobolev type inequality on $\mathbb{G}$:
		$$
		\left\|f\right\|_{L^{p}(\mathbb{G})}\leq\frac{p}{Q}\left\|\mathbb{E} f\right\|_{L^{p}(\mathbb{G})}, \quad 1<p<\infty,
		$$
		where the constant $\frac{p}{Q}$ is sharp.
		
		In the Euclidean case $\mathbb{G}=(\mathbb{R}^{n};+)$, we have $Q=n$, so for $n\geq1$ it implies the Sobolev type inequality:
		$$
		\left\|f\right\|_{L^{p}(\mathbb{R}^{n})}\leq\frac{p}{n}\left\|x \cdot \nabla f\right\|_{L^{p}(\mathbb{R}^{n})}, \quad 1<p<\infty.
		$$

	Above ideas can be extended to other weighted identities on $L^2(\mathbb G)$: e.g.
		for all $k\in\mathbb N$ and $\alpha\in\mathbb{R}$, for every $f\in C_{0}^{\infty}(\mathbb{G}\backslash\{0\}), \alpha\in \mathbb{R}$, and any homogeneous quasi-norm $|\cdot|$ on $\mathbb{G}$ we have:

		\begin{multline*}
		{ \left\|\frac{1}{|x|^{\alpha}}\mathcal{R}_{|x|}^{k}f\right\|^{2}_{L^{2}(\mathbb{G})}=
			\left[\prod_{j=0}^{k-1}
			\left(\frac{Q-2}{2}-(\alpha+j)\right)^{2}\right]\left\|\frac{f}{|x|^{k+\alpha}}
			\right\|^{2}_{L^{2}(\mathbb{G})}}
		\\+\sum_{l=1}^{k-1}\left[\prod_{j=0}^{l-1}
		\left(\frac{Q-2}{2}-(\alpha+j)\right)^{2}\right]\left\|\frac{1}{|x|^{l+\alpha}}\mathcal{R}_{|x|} ^{k-l}f+
		\frac{Q-2(l+1+\alpha)}{2|x|^{l+1+\alpha}}\mathcal{R}_{|x|} ^{k-l-1}f\right\|^{2}_{L^{2}(\mathbb{G})}
		\\
		+\left\|\frac{1}{|x|^{\alpha}}\mathcal{R}_{|x|} ^{k}f+\frac{Q-2-2\alpha}{2|x|^{1+\alpha}}\mathcal{R}_{|x|} ^{k-1}f \right\|^{2}_{L^{2}(\mathbb{G})}.
		\end{multline*}

		\begin{equation*}
		{\left\|\frac{1}{|x|^{\alpha}}\mathbb{E} f\right\|^{2}_{L^{2}(\mathbb{G})}=
			\left(\frac{Q}{2}-\alpha\right)^{2}
			\left\|\frac{f}{|x|^{\alpha}}\right\|^{2}_{L^{2}(\mathbb{G})}}
		+\left\|\frac{1}{|x|^{\alpha}}\mathbb{E} f+\frac{Q-2\alpha}{2|x|^{\alpha}}f
		\right\|^{2}_{L^{2}(\mathbb{G})}.
		\end{equation*}

The following lemma allows to obtain fractional orders of previous inequalities.
		\begin{lem}
			The operator $\mathbb{A}=\mathbb{E}\mathbb{E}^{*}$ is Komatsu-non-negative in $L^{2}(\mathbb{G})$:
			\begin{equation}\label{Komatsu0}
			\quad \|(\lambda+\mathbb{A})^{-1}\|_{L^{2}(\mathbb{G})\rightarrow L^{2}(\mathbb{G})}\leq \lambda^{-1},\; \forall\lambda>0.
			\end{equation}
		\end{lem}
		Since $\mathbb{A}$ is Komatsu-non-negative, we can define fractional powers of the operator $\mathbb{A}$ as in \cite{Martinez-Sanz:fractional} and we denote
		$$|\mathbb{E}|^{\beta}=\mathbb{A}^{\frac{\beta}{2}}, \;\;\;\;\;\beta\in \mathbb{C}.$$

	For example, we have the following Hardy inequality with the fractional Euler operator. 
		\begin{thm} \cite{RSY18}
			Let $\mathbb{G}$ be a homogeneous group
			of homogeneous dimension $Q$, $\beta\in \mathbb{C_{+}}$ and let $k>\frac{{\rm Re} \beta}{2}$ be a positive integer. Then for all complex-valued functions $f\in C_{0}^{\infty}(\mathbb{G}\backslash\{0\})$ we have
			\begin{equation}\label{fracineq1}
			\left\|f\right\|_{L^{2}(\mathbb{G})}\leq C\left(k-\frac{\beta}{2},k\right)\left(\frac{2}{Q}\right)^{{\rm Re}\beta}\left\|\mathbb{|E|}^{\beta} f\right\|_{L^{2}(\mathbb{G})}, \; Q\geq1,
			\end{equation}
			where \begin{equation}\label{fracineq2}C(\beta,k)=\frac{\Gamma(k+1)}{|\Gamma(\beta)\Gamma(k-\beta)|}\frac{2^{k-{\rm Re}\beta}}{{\rm Re}\beta (k-{\rm Re}\beta)}.
			\end{equation}
		\end{thm}

{\bf Stein-Weiss type inequalities on  homogeneous groups:}
Let $0<\lambda<Q$ and $\frac{1}{q}=\frac{1}{p}+\frac{\lambda}{Q}-1$ with $1<p<q<\infty$. Then the following inequality is valid on $\mathbb{G}$ of homogeneous dimension $Q$:
	\begin{equation}\label{EQ:HLSi1}
	\left|\int_{\mathbb{G}}\int_{\mathbb{G}}\frac{f(y)h(x)}{|y^{-1} x|^{\lambda}}dxdy\right|\leq C\|f\|_{L^{p}(\mathbb{G})}\|h\|_{L^{q'}(\mathbb{G})},
	\end{equation}
for all $f\in L^{p}(\mathbb{G})$ and $h\in L^{q'}(\mathbb{G})$. The Euclidean version of this inequality is called the Hardy-Littlewood-Sobolev (HLS) inequality. In 1958, Stein and Weiss established a two-weight extension of the (Euclidean) HLS inequality \cite{StWe58} (see also \cite{FLieb12}). Nowadays, the two-weight extension of the HLS inequality is called the Stein-Weiss inequality. 
	Note that Folland and Stein obtained the HLS inequality on the Heisenberg groups \cite{FS74}.
	On stratified groups a version of the Stein-Weiss inequality was obtained in \cite{GMS}. The Stein-Weiss inequality was extended to graded groups in \cite{RY}:

\begin{thm}\cite{RY}
	Let $\mathbb{G}$ be a graded group of homogeneous dimension $Q$ and let $|\cdot|$ be an arbitrary homogeneous quasi-norm. Let $1<p,q<\infty$, $0\leq a<Q/p$ and $0\leq b<Q/q$. Let $0<\lambda<Q$, $0\leq \alpha <a+Q/p'$ and $0\leq \beta\leq b$ be such that $$(Q-ap)/(pQ)+(Q-q(b-\beta))/(qQ)+(\alpha+\lambda)/Q=2$$ and $\alpha+\lambda\leq Q$, where $1/p+1/p'=1$. Then for all $f\in \dot{L}^{p}_{a}(\mathbb{G})$ and $h\in \dot{L}^{q}_{b}(\mathbb{G})$ we have 
	\begin{equation}\label{HLS_ineq1_grad}
	\left|\int_{\mathbb{G}}\int_{\mathbb{G}}\frac{\overline{f(x)}h(y)}{|x|^{\alpha}|y^{-1}x|^{\lambda}|y|^{\beta}}dxdy\right|\leq C\|f\|_{\dot{L}^{p}_{a}(\mathbb{G})}\|h\|_{\dot{L}^{q}_{b}(\mathbb{G})}
	\end{equation}
	where $C$ is a positive constant independent of $f$ and $h$. 
	Here $\dot{L}^{p}_{a}(\mathbb{G})$ stands for a homogeneous Sobolev space of order $a$ over $L^p$ on the graded Lie group $\mathbb{G}.$
\end{thm}

In \cite{KRS19},	the Stein-Weiss inequality was extended to general homogeneous groups.
	
\begin{thm}\cite{KRS19}  Let $|\cdot|$ be an arbitrary homogeneous quasi-norm on $\mathbb{G}$ of homogeneous dimension $Q$.
	Let $$0<\lambda<Q,\quad 1<p\leq q<\infty,$$ $$\alpha<\frac{Q}{p'},\quad \beta<\frac{Q}{q},\quad \alpha+\beta\geq0,$$ $$\frac{1}{q}=\frac{1}{p}+\frac{\alpha+\beta+\lambda}{Q}-1,\quad \frac{1}{p}+\frac{1}{p'}=1, \quad \frac{1}{q}+\frac{1}{q'}=1.$$ Then for all $f\in L^{p}(\mathbb{G})$ and $h\in L^{q'}(\mathbb{G})$ we have 
	\begin{equation}\label{EQ:HLSi2}
	\left|\int_{\mathbb{G}}\frac{f(y)h(x)}{|x|^{\beta}|y^{-1}x|^{\lambda}|y|^{\alpha}}dxdy\right|\leq C\|f\|_{L^{p}(\mathbb{G})}\|h\|_{L^{q'}(\mathbb{G})},
	\end{equation}
	where $C$ is a positive constant independent of $f$ and $h$.
	\end{thm}

The reverse Stein-Weiss inequality also holds on $\mathbb{G}$.
\begin{thm}\label{stein-weiss3} \cite{KRS19b}
	Let $\mathbb{G}$ be a homogeneous group of homogeneous dimension $Q$ and let $|\cdot|$ be an arbitrary homogeneous quasi-norm on $\mathbb{G}$.
	Let 
	
	 $$0\leq\alpha<-\frac{Q}{q}\quad 0\leq\beta<-\frac{Q}{p'},\quad q,p'\in(0,1),$$
	
	$$\frac{1}{q'}+\frac{1}{p}=\frac{\alpha+\beta+\lambda}{Q}+2,\; 0<\lambda<\infty,\;\frac{1}{p}+\frac{1}{p'}=1,\;\frac{1}{q}+\frac{1}{q'}=1.$$ Then for all non-negative functions $f\in L^{q'}(\mathbb{G})$ and $h\in L^{p}(\mathbb{G})$ we have
	\begin{equation}\label{stein-weiss}
	\int_{\mathbb{G}}\int_{\mathbb{G}}|x|^{\alpha}|y^{-1}x|^{\lambda}f(x)h(y)|y|^{\beta}dxdy\geq C\|f\|_{L^{q'}(\mathbb{G})}\|h\|_{L^{p}(\mathbb{G})},
	\end{equation}
	where $C$ is a positive constant independent of $f$ and $h$.
	
\end{thm}

	In the Euclidean (Abelian) case ${\mathbb G}=(\mathbb R^{N},+)$, hence $Q=N$ and $|\cdot|$ can be any homogeneous quasi-norm  on $\mathbb R^{N}$,  particularly with the standard Euclidean distance, that is, with $|\cdot|=\|\cdot\|_{E}$ it was studied in \cite{CLT}.

An improved version of the reverse Stein-Weiss inequality can be also proved \cite{KRS19b}:
 Let $|\cdot|$ be an arbitrary homogeneous quasi-norm on $\mathbb{G}$ of homogeneous dimension $Q$.
	Let $\frac{1}{q'}+\frac{1}{p}=\frac{\alpha+\beta+\lambda}{Q}+2$ with $p,q'\in(0,1)$, where $0<\lambda<\infty,$ $\frac{1}{p}+\frac{1}{p'}=1$  and $\frac{1}{q}+\frac{1}{q'}=1$. Then for all non-negative functions $f\in L^{q'}(\mathbb{G})$ and $h\in L^{p}(\mathbb{G})$, inequality \eqref{stein-weiss} holds, that is,
	\begin{equation*}
	\int_{\mathbb{G}}\int_{\mathbb{G}}|x|^{\alpha}|y^{-1}x|^{\lambda}f(x)h(y)|y|^{\beta}dxdy\geq C\|f\|_{L^{q'}(\mathbb{G})}\|h\|_{L^{p}(\mathbb{G})}
	\end{equation*}
	if   $0\leq\alpha<-\frac{Q}{q}$ or	$0\leq\beta<-\frac{Q}{p'}$.


\begin{thebibliography}{NZW01}


\bibitem{Adimurthi} Adimurthi and A. Sekar.
\newblock Role of the fundamental solution in Hardy-Sobolev-type inequalities.
\newblock {\em Proc. Roy. Soc. Edinburgh Sect. A}, 136(6):1111-1130, 2006.



\bibitem{BT02} N. Badiale and G. Tarantello.
\newblock A Sobolev-Hardy inequality with applications to a nonlinear elliptic equation arising in astrophysics.
\newblock {\em Arch. Ration. Mech. Anal.}, 163:259-293, 2002.

\bibitem{BCG06} H. Bahouri, J.-Y. Chemin, and I. Gallagher.
\newblock Refined Hardy inequalities.
\newblock {\em Sup. Pisa Cl. Sci}, 5:375-391, 2006.


\bibitem{BCX07} H. Bahouri, J. Y. Chemin, and C. J. Xu. 
\newblock Trace and trace lifting theorems in weighted Sobolev spaces.
\newblock {\em Journal of the Institute of Mathematics of Jussieu}, 4(4):509-552, 2007.


	\bibitem{BEL15}
A. A. Balinsky, W. D. Evans, and R. T. Lewis.
\newblock The analysis and geometry of Hardy's inequality.
\newblock Universitext. Springer, Cham, 2015.

 
 	\bibitem{CDPT}
 L. Capogna, D. Danielli, S. D. Pauls, and J. T. Tyson.
	\newblock An introduction to the Heisenberg group and the
sub-Riemannian isoperimetric problem.
	\newblock Progress in Mathematics, Vol. 314, Birkh\"auser, 2007. 


\bibitem{Ambrosio04} L. D'Ambrosio.
\newblock Some Hardy inequalities on the Heisenberg group.
\newblock {\em Differential Equations}, 40(4):552-564, 2004.

\bibitem{Ambrosio05} L. D'Ambrosio.
\newblock Hardy-type inequalities related to degenerate elliptic differential operators.
\newblock {\em Ann. Sc. Norm. Super. Pisa Cl. Sci. (5)}, 4(3):451-486, 2005.



\bibitem{CLT}
L.~Chen, G.~Lu and C.~Tao.  Reverse Stein-Weiss Inequalities on the Upper Half Space and the Existence of Their Extremals. 
{\em Advanced Nonlinear Studies}, 19(3):475-494, 2019.

\bibitem{DGP11} D. Danielli, N. Garofalo, and N. C. Phuc.
\newblock Hardy-Sobolev type inequalities with sharp constants in Carnot-Carath\'eodory spaces.
\newblock {\em Potential Anal.}, 34:223-242, 2011.

	\bibitem{Davies}
E.B. Davies. 
\newblock The Maz'ya anniversary Collection.
\newblock OT Adv. Appl. 55-67 (1999)


	\bibitem{Dyer}
J.L. Dyer. 
\newblock A nilpotent Lie algebra with nilpotent automorphism group.
\newblock {\em  Bull. Amer. Math. Soc.}, 76:52-56, 1970



\bibitem{FMT07}
S. Filippas, V. Maz'ya,  A. Tertikas.
\newblock Critical Hardy-Sobolev inequalities.
\newblock {\em J. Math. Pures Appl.}, 87:37-56, 2007.

	\bibitem{FR}
V. Fischer,	M.~Ruzhansky.
	\newblock Quantization on nilpotent Lie groups.
	\newblock Progress in Mathematics, Vol. 314, Birkh\"auser, 2016. 

\bibitem{Folland75}
G. Folland.
\newblock Subelliptic estimates and function spaces on nilpotent Lie groups.
\newblock {\em Arkiv for Mat..}, 13 (2): 161-207, 1975.

\bibitem{Folland19}
G. Folland.
\newblock The Heisenberg group and its relatives in the work of Elias M. Stein.
\newblock {\em J Geom Anal.}, 13 (2): 161-207, 2019.

	\bibitem{FS74}
G.~B. Folland and E.~M. Stein.
\newblock Estimates for the $\overline{\partial_{b}}$ complex and
analysis on the Heisenberg group.
\newblock {\em Comm. Pure Appl. Math.}, 27:429--522, 1974.

\bibitem{FS1}
G. B.~Folland and E. M.~Stein.
\newblock Hardy Spaces on Homogeneous Groups.
\newblock Mathematical Notes, Vol. 28, Princeton University Press, Princeton, N.J.; University of
Tokyo Press, Tokyo, 1982.


\bibitem{FranPra19}
V. Franceschi, D. Prandi,
\newblock Hardy-type inequalities for the Carnot-Carath\'{e}odory distance in the Heisenberg group.
\newblock {\em J Geom Anal}, https://doi.org/10.1007/s12220-020-00360-y, 2020.
arXiv:1903.08486, 2019.


\bibitem{FL12}
R.L. Frank,  M. Loss.
\newblock The Hardy-Sobolev-Maz'ya inequalities for arbitrary domains.
\newblock {\em J. Math. Pures Appl.}, 97:39-54, 2012.

\bibitem{FLieb12}
R.~L.~Frank and E.~H~Lieb.
\newblock Sharp constants in several inequalities on
the Heisenberg group.
\newblock {\em Ann. of Math.}, 176:349--381, 2012.

\bibitem{Garofalo}
N. Garofalo.
\newblock Geometric second derivative estimates in Carnot groups and convexity.
\newblock {\em Manuscripta Math.}, 126:353-373, 2008.

\bibitem{GL} N.~Garofalo and E.~Lanconelli.
\newblock Frequency functions on the Heisenberg group, the uncertainty principle and unique continuation.
\newblock {\em Ann. Inst. Fourier (Grenoble)}, 40:313-356, 1990.


\bibitem{GK08}
J. A. Goldstein and I. Kombe.
\newblock The Hardy inequality and non-linear parabolic equations on Carnot groups.
\newblock {\em  Nonlinear Anal.}, 69(12):4643-4653, 2008.


\bibitem{Hardy}
G. H. Hardy.
\newblock Notes on some points in the integral calculus (51).
\newblock {\em  Messenger of Math.}, 48:107-112, 1918.


\bibitem{Hor}
L. H\"ormander.
\newblock Hypoelliptic second order differential equations.
\newblock {\em  Acta Math.}, 119:147-171, 1967.

\bibitem{JS11}
Y. Jin and S. Shen.
\newblock Weighted Hardy and Rellich inequality on Carnot groups.
\newblock {\em  Arch. Math. (Basel)},  96(3):263-271, 2011.

 \bibitem{GMS}
V.~Guliyev, R.~Mustafayev and A.~Serbetci.
\newblock Stein-Weiss inequalities for the fractional integral operators in Carnot groups and applications.
\newblock{\em Complex Variables and Elliptic Equations}, 55:8--10, 847--863, 2010.

	\bibitem{KRS19}
		 A. Kassymov,  M. Ruzhansky, and D. Suragan.
		\newblock  Hardy-Littlewood-Sobolev and Stein-Weiss inequalities on homogeneous Lie groups,
		\newblock {\em Integral Transforms and Special Functions}, 30:643-655, 2019.

	\bibitem{KRS19b}
 A. Kassymov,  M. Ruzhansky, D. Suragan. Reverse Stein-Weiss, Hardy-Littlewood-Sobolev, Hardy, Sobolev and Caffarelli-Kohn-Nirenberg inequalities on homogeneous groups
	{\em  arXiv:1912.01460}, 2019.

\bibitem{Larson}
S. Larson.
\newblock Geometric Hardy inequalities for the sub-elliptic Laplacian on convex domain in the Heisenberg group.
\newblock {\em  Bull. Math. Sci.},  6:335-352, 2016.

\bibitem{Lehrb}
J. Lehrb\"ack.
\newblock Hardy inequalities and Assouad dimensions.
\newblock {\em  Journal d'Analyse Mathematique},  131 (1):367-398, 2017.


\bibitem{LY08}
J. Luan, Q. Yang.
\newblock A Hardy type inequality in the half-space on $R^n$ and Heisenberg group.
\newblock {\em  J. Math. Anal. Appl.},  347:645-651, 2008.


\bibitem{Leray}
J. Leray.
\newblock \'Etude de diverses \'equations int\'egrales non lin\'eaires et de
quelques problemes que pose l'hydrodynamique.
\newblock {\em  J. Math. Pures Appl.}, 12:1-82, 1933.

\bibitem{Martinez-Sanz:fractional}
C.~C.~Martinez and M.~A.~Sanz.
\newblock {\em The theory of fractional powers of operators,}
\newblock {\em North-Holland Mathematics Studies.} North-Holland Publishing Co., Amsterdam, 2001.

	\bibitem{Mazya}
V.G. Maz'ya. 
\newblock  Sobolev Spaces.
\newblock Springer, Berlin, 1985. 



\bibitem{NLNJMSLJ}
D. T. Nguyen, N. Lam-Hoang, and T. A. Nguyen.
\newblock Hardy and Rellich inequalities with exact missing terms on homogeneous groups.
\newblock {\em J. Math. Soc. Japan}, 71(4):1243-1256, 2019.

\bibitem{NLNArchiv}
D. T. Nguyen, N. Lam-Hoang, and T. A. Nguyen.
\newblock Hardy-Rellich identities with Bessel pairs.
\newblock {\em Archiv der Mathematik}, 113:95-112, 2019.


\bibitem{NZW}
P. Niu, H. Zhang, and Y. Wang.
\newblock Hardy type and Rellich type inequalities on the Heisenberg group.
\newblock {\em Proc. Amer. Math. Soc.}, 129(12):3623-3630, 2001.

	\bibitem{Opic}
B. Opic,  A. Kufner. 
\newblock Hardy-type inequalities.
\newblock Pitman Research Notes, 219, 1990.


\bibitem{OS19}
T. Ozawa  and D. Suragan.
\newblock Representation formulae for the higher order Steklov and $L^{2^{m}}$-Friedrichs inequalities.
\newblock {\em preprint}, 2019.



\bibitem{RS17}
	M. Ruzhansky  and D. Suragan.
\newblock Hardy and Rellich inequalities, identities, and sharp remainders on homogeneous groups.
\newblock {\em Adv. Math.}, 317:799-822, 2017.


\bibitem{RS17Kac}
M. Ruzhansky  and D. Suragan.
\newblock Layer potentials, Kac's problem, and refined Hardy inequality on homogeneous Carnot groups.
\newblock {\em Adv. Math.}, 308:483-528, 2017.


\bibitem{Ruzhansky-Suragan:JDE}
 M. Ruzhansky and D. Suragan. 
On horizontal Hardy, Rellich, Caffarelli-Kohn-Nirenberg and $p$-sub-Laplacian inequalities on stratified groups.
{\em J. Differential Equations}, 262:1799-1821, 2017.

\bibitem{RSY:18}
M.~Ruzhansky, D.~Suragan and N.~Yessirkegenov.
\newblock Sobolev type inequalities, Euler-Hilbert-Sobolev and Sobolev-Lorentz-Zygmund spaces on homogeneous groups.
\newblock {\em Integral Equations and Operator Theory}, 90:10, 2018.

\bibitem{RSY18_Tran}
M.~Ruzhansky, D.~Suragan and N.~Yessirkegenov. Extended Caffarelli-Kohn-Nirenberg inequalities, and remainders, stability, and superweights for $L_p$-weighted Hardy inequalities
\newblock {\em Trans. Amer. Math. Soc. Ser. B.}, 5:32-62, 2018.



\bibitem{RSS18}
M. Ruzhansky, B. Sabitbek,  and D. Suragan. 
Geometric Hardy and Hardy-Sobolev inequalities on Heisenberg groups.
{\em to appear in Bull. Math. Sci.}, arXiv:1811.07181, 2018.

\bibitem{RSCommunContempMath}
M. Ruzhansky and D. Suragan. 
Anisotropic L2-weighted Hardy and L2-Caffarelli-Kohn-Nirenberg inequalities.
{\em Commun. Contemp. Math.}, 19, no. 6, 1750014, 2017.

\bibitem{RSY18}
M. Ruzhansky, D. Suragan, and N. Yessirkegenov.
\newblock  Hardy-Littlewood, Bessel-Riesz, and fractional integral operators in anisotropic Morrey and Campanato spaces,
\newblock {\em Fract. Calc. Appl. Anal.}, 21(3):577-612, 2018.

\bibitem{RS_book}
	M. Ruzhansky  and D. Suragan. 
	\newblock Hardy inequalities on homogeneous groups.
	\newblock \textit{Progress in Math.} Vol. 327, Birkh\"auser, 588 pp, 2019.

\bibitem{RY}
M.~Ruzhansky and N.~Yessirkegenov.
\newblock Hypoelliptic functional inequalities.
\newblock {\em  arXiv:1805.01064v1}, 2018.

\bibitem{Stein}
E. M. Stein.
\newblock Some problems in harmonic analysis suggested by symmetric spaces and semi-simple groups.
\newblock {\em  Actes, Congres Intern. Math., Nice}, 1:179-189, 1970.

\bibitem{SSW03}
S.~Secchi, D.~Smets and M.~Willen.
\newblock Remarks on a Hardy-Sobolev inequality.
\newblock {\em C. R. Acad. Sci. Paris, Ser. I}, 336:811-815, 2003.


\bibitem{StWe58}
E.~M.~Stein and G.~Weiss.
\newblock Fractional integrals on $n$-dimensional Euclidean Space.
\newblock {\em J. Math. Mech.}, 7(4):503--514, 1958.

\bibitem{Yang13}
Q.-H.~Yang.
\newblock Hardy type inequalities related to Carnot-Carath\'{e}odory distance on the Heisenberg group.
\newblock {\em Proc. Amer. Math. Soc.}, 141:351--362, 2013.


\end{thebibliography}
\end{document}